\documentclass[12pt]{amsart}

\usepackage{amsmath,amsthm,amssymb,mathrsfs,amsfonts,verbatim,enumitem,color}
\usepackage{bbm,blindtext}

\usepackage{mathabx}
\usepackage{etoolbox} 
\usepackage{bbm}
\usepackage[all,tips]{xy}
\usepackage{graphicx,ifpdf}
\ifpdf
\fi
\usepackage{geometry}
\usepackage{lmodern}

\usepackage{scrextend}
\KOMAoption{fontsize}{13pt}

\usepackage{blindtext}
\newtheorem{thm}{Theorem}[section]

\theoremstyle{definition}

\newtheorem{ex}[thm]{Example}

\theoremstyle{remark}
\newtheorem{rem}[thm]{Remark}

\numberwithin{equation}{section}

\makeatletter
\newcommand{\rmnum}[1]{\romannumeral #1}
\newcommand{\Rmnum}[1]{\expandafter\@slowromancap\romannumeral #1@}
\makeatother

\newcommand{\bx}{\mathbf{x}}
\newcommand{\by}{\mathbf{y}}

\newcommand{\R}{\mathbb{R}}
\newcommand{\N}{\mathbb{N}}
\newcommand{\SL}{\operatorname{SL}}
\newcommand{\GL}{\operatorname{GL}}
\newcommand{\Z}{\mathbb{Z}}

\newcommand{\til}{\widetilde}
\newcommand{\vre}{\varepsilon}
\newcommand{\dd}{\; \mathrm{d}}

\newcommand{\Ad}{\operatorname{Ad}}

\begin{document}
	
	\title{Quantitative Multiple pointwise convergence and effective multiple correlations}

	% Information for first author
	%\author{}
	
	\author{Rene R\"uhr}
		\address{Faculty of Mathematics, Technion, Haifa, Israel}
	\email{rener@campus.technion.ac.il}
	\thanks{R.R.\ acknowledges that this project has received
funding from the European Research Council (ERC) under the European Union's Horizon 2020 research
and innovation program (grant agreement No. 754475).}
	% General info
		
	\author{Ronggang Shi}
	% Address of record for the research reported here
	\address{Shanghai Center for Mathematical Sciences, Jiangwan Campus, Fudan University, No.2005 Songhu Road, Shanghai, 200433, China}
	% Current address
	%\curraddr{Department of Mathematics, Ohio State
	%University, Columbus, Ohio 43210}
	\email{ronggang@fudan.edu.cn}
	% \thanks will become a 1st page footnote.
	\thanks{R.S.\ is supported by NSFC 11871158.}

	\date{}
	
	%\dedicatory{This paper is dedicated to my advisor.}
	

	\begin{abstract}
		We show that effective $2\ell$-multiple correlations imply quantitative  $\ell$-multiple pointwise  ergodic theorems. The result has a wide class of applications which include  subgroup actions on homogeneous spaces, ergodic nilmanifold  automorphisms, subshifts of finite type
		% endowed  with Gibbs measures 
		and Young towers. 
	\end{abstract}

	\maketitle
	\markright{}

	\section{Introduction}
	Let $h: X\to X$ be a measure preserving  map
	on a probability space  $(X, \mathcal X, \mu) $.  
	The $\ell$-multiple pointwise ergodic theorem 
	%for the probability measure preserving system	 $(X,\mathcal X,  \mu, h)$ 
	 asks about the almost sure convergence of the average
	\begin{align}\label{eq;fur}
	\frac{1}{N} \sum _{n=1}^N f_1(h^n x)f_2(h^{2n}x)\cdots f_\ell (h^{\ell n}x)\qquad (N\to \infty),
	\end{align}
	where $f_i:X\to \R\ (1\le i\le \ell)$ are bounded measurable functions. When $\ell=1$, the almost sure convergence of (\ref{eq;fur}) is the Birkhoff ergodic theorem.
	The study of the question for $\ell=2$ is initiated by Hillel Furstenberg \cite[p.~96]{f}. Furstenberg's question was answered affirmatively by Bourgain \cite{b}. 
	
	Derrien and Lesigne \cite{dl} showed that   (\ref{eq;fur})
	converges almost surely for all $f_i \in L^\infty(X, \mathcal X, \mu)$ if and only if the same holds 
	for all $f_i\in L^\infty(X, \mathcal P, \mu)$ where $\mathcal P$ is the Pinsker $\sigma$-algebra of $h$.
	In this paper all the function spaces such as $L^p, C_c^\infty$ and $ C^\theta$ consist of real valued functions.  Therefore, the almost sure convergence of  (\ref{eq;fur}) holds if $h$ is a Kolmogorov automorphism. Since every Bernoulli automorphism is a Kolmogorov automorphism, the same result holds if  $h$ is a two-sided Bernoulli shift. 
	The same method was used by 
	Assani \cite{assani} to show the  almost sure convergence of (\ref{eq;fur}) when $h$ is weakly mixing and the restriction of $h$ to $(X, \mathcal P, \mu)$ has singular spectrum.  
	  We refer the readers to Walters \cite
%	  [\S 4.9\S 4.10]
	  {walters} for the definitions of these
	 concepts.

Recently, progress on higher order pointwise convergence was obtained by  Huang-Shao-Ye \cite{hsy} for  ergodic distal systems  and by  Gutman-Huang-Shao-Ye \cite{ghsy} for 
	weakly mixing pairwise independent systems. 
	See also Donoso-Sun \cite{ds1}\cite{ds2} for an extension of \cite{hsy} to commuting distal transformations.

	We remark here that 
	the multiple mean  ergodic theorem for  commuting maps  are well-understood, see Tao \cite{tao} and references there.  An interesting extension of the mean ergodic theorem to nilpotent group actions with polynomial ergodic averages was obtained by Walsh \cite{walsh}. The pointwise  convergence of polynomial ergodic averages is known for $\ell=1$ by Bourgain \cite{bourgain1}\cite{bourgain} and 
	for Kolmogorov automorphisms by \cite{dl}.

	In a recent paper \cite{shi}, the authors noticed that effective multiple correlations can be used to prove quantitative multiple  pointwise  convergence. 
	  The aim of this paper is to develop this method further in an abstract setting (Theorem \ref{thm;final}) and obtain more results on quantitative multiple pointwise convergence. The examples contain commuting subgroup  actions on homogeneous spaces, nilmanifold
	  automorphisms,   subshifts of finite type 
	  and Young towers.

	  We begin with subgroup actions on homogeneous spaces.
	  A unitary representation  of a locally compact topological group  $H$ on a (real) Hilbert space $V$ is said
	  to have a spectral gap if there exists $\delta >0$
	  and a compactly supported probability measure $\nu$ on $H$ such that 
	  \[
	\left  \| \int_ H hv  \dd \nu(h)\right \|\le (1-\delta) \|v\| \quad \mbox{ for all }v\in V. 
	  \]
	  
	  Let $\Gamma$ be a lattice of a connected Lie group $G$. 
	  The homogeneous space $X=G/\Gamma$ carries a unique $G$ invariant probability measure $\mu$. 
   Let $H$ be a connected semi-simple subgroup of $G$ with finite center. The  natural left translation action of 
	  $H$ on $X$ induces a unitary representation of 
	  $H$ on $$L^2_{ 0}(X, \mu)=\{ f\in L^2(X, \mu): \mu(f)=0 \}\quad \mbox{via}\quad  (hf)(x)=f(h^{-1}x).$$ 
We say the action of $H$ on $X$ has a strong spectral gap if the representation of each noncompact simple factor of $H$ on $L^2_0(X, \mu)$ has a spectral gap. 
A sequence  of   integers $\{r_n \}$ is said to be non-clustered if
the cardinality of  $\{n\in \N : r_n=m \}$ is uniformly bounded for all $m\in \Z$. 
In this paper we adopt the convention that the set of natural numbers 
 $\N$ does not contain $0$. 
For $h\in H$, we let 
$\langle h \rangle$ be the group generated by $h$.  We use $C_c^\infty (X)$ to denote the 
space of compactly supported and  smooth functions on $X$. We call a group element $h\in H$ quasi-unipotent if
all the eigenvalues of 
 its corresponding adjoint action $\Ad(h)\in \GL(\mathrm{Lie}(H))$ have  modulus one.
	 \begin{thm}
	 	\label{thm;unipotent}
	 Let $G, \Gamma , X,\mu, H$ be as above.
%Let $\Gamma$ be a lattice of a connected Lie group $G$. Let $H$ be a connected semi-simple subgroup of $G$ with finite center. 
Suppose the action of $H$ on $X$ has a strong spectral gap. 
Let $h_1, \ldots, h_\ell$ be commuting elements of $H$  such that all  the groups  $\langle h_i\rangle $ and   $\langle h_i h_{j}^{-1}\rangle $
 $(i\neq j)$ are unbounded. Then there exists $\varepsilon>0$ such that  for any $\ell \in \N$, any $f_1, \ldots, f_\ell\in C_c^\infty(X)$ and any non-clustered   sequence of integers  $\{ r_n \}$,  one has 
 $\mu$-almost surely 
\begin{align}\label{eq;unipotent}
\frac{1}{N} \sum_{n=1}^N f_1(h_1^{r_n} x)\cdots f_\ell(h_\ell^{r_n}x)=\mu(f_1)\cdots  \mu(f_\ell)+    o_x(N^{-\varepsilon}).
\end{align}
Moreover, if all the $h_i$ and $h_ih_j^{-1}\ (i\neq j)$  are not quasi-unipotent, then (\ref{eq;unipotent}) holds almost surely with a better error term 
\begin{align}\label{eq;fuwocheng}
o_x(N^{-\frac{1}{2}}
\log ^{\frac{3}{2}+\varepsilon}N)\qquad \forall \varepsilon>0.
\end{align}
	 \end{thm}

 Here the notation $\varphi(N)=o(\psi(N))$ means $\lim_{ N\to \infty}\frac{\varphi(N)}{\psi(N)}=0$.
   In this paper the phrase quantitative is used to denote a weaker version of effectiveness in the sense that we do not have  any control of  the convergence speed in the little $o$ notation appearing in the ergodic theorems.
 Theorem \ref{thm;unipotent} is already interesting in the simplest case where $\ell=1$ as can be seen from 
 the following   example.

\begin{ex}\label{ex;polynomial}
	Theorem \ref{thm;unipotent}  holds for  $H=G=\SL_2(\R)$,  $\Gamma=\SL_2(\Z)$, $\ell =1$,  $h_1=u(s)=\left(
\begin{array}{cc}
1 & s\\
0 & 1
\end{array}
	\right)$ where $s\neq 0$ and $r_n=p(n)$ where $p(n)$ is a  
polynomial with integer coefficients and degree at least two.
 The qualitative result in this case is known earlier by \cite{bourgain1}. According to a conjecture of Shah \cite{shah}, it is expected that
the qualitative result holds for all the generic $x$, where $x$ is said to be generic if the trajectory $\{u(t)x: t\ge 0 \}$ is dense in $X$.   

Another interesting case is to take  
  $r_n$ to be  the $n$-th 
prime number. This type of questions was considered earlier by Sarnak and Ubis \cite{su}. 
It is also expected that the qualitative result holds for all the generic $x$. 
\end{ex}

Next we consider the multiple ergodic theorem for a fixed 
  probability measure preserving 
system $(X, \mathcal X, \mu, h)$. 
Usually when $X$ is a topological space, we take $\mathcal X$ to be the Borel $\sigma$-algebra. In this case or other cases where $\mathcal X$ is understood or 
does not play an important role, 
we may omit it  in our notation. 
   We say $(X, \mu,  h)$ is
 $k$-mixing with  rate $n^{-\delta}$ ($\delta>0$) for  a vector space $\mathcal F $ of measurable  functions  on $X$
 if  for any $f_0, f_1, \ldots, f_k\in \mathcal F$ and any pairwise distinct  $n_0, n_1, \ldots, n_k\in \N$, one has
 \begin{align}\label{eq;letter}
 \int_X f_0(h^{n_0}x)\cdots f_k(h^{n_k}x)\dd\mu =\mu(f_0)\cdots \mu(f_k)
 +O( \inf_{i\neq j}|n_i-n_j|^{-\delta}), 
 \end{align}
 where the implied constant is allowed to depend on $k$ and $f_1, \ldots, f_k$. Here $\varphi=O(\psi)$ 
 for two functions means $|\varphi|\le C |\psi|$ for some constant $C>0$. 
 Our definition of $k$-mixing follows the convention of Glasner \cite[Chapter 3 \S 9]{glasner}.
 The pairwise distinctness assumption for $ n_1, \ldots, n_k\in \N$ is not essential, since we could drop it and  and replace  $|n_i-n_j|^{-\delta}$ by $ (|n_i-n_j|+1)^{-\delta}$.
 
%Sometimes in the statement of multiple exponential mixing,  it  is assumed that $n_i\in \N$,   which is equivalent to our statement since the measure  $\mu$ is $h$ invariant. 

 \begin{thm}
 	\label{thm;mixing}
 	If a probability measure preserving system $(X, \mathcal X, \mu, h)$ is $(2\ell-1)$-mixing
 	 with  rate $n^{-\delta}$ ($\delta>0$) for  a  vector space of measurable functions $\mathcal F $ containing constant functions, then 
 	for any pairwise distinct  $m_1, \ldots, m_\ell\in \N $, any $f_1, \ldots, f_\ell\in \mathcal F$ and any non-clustered sequence 
 	of positive  integers $\{r_n \}$,  one has $\mu$-almost surely
 	\begin{align}\label{eq;repeat}
 	\frac{1}{N} \sum_{n=1}^N f_1(h^{m_1r_n} x)\cdots f_\ell(h^{m_\ell r_n }x)=\mu(f_1)\cdots \mu(f_\ell)+ o_x(\rho_{\varepsilon, \delta}(N)), 
 	\end{align}
 	where 
 $\rho_{\varepsilon, \delta}(N)= N^{-\frac{1 }{2}}\log^{\frac{3}{2}+\varepsilon} N$ if $\delta >1$ and 
 $\rho_{\varepsilon,\delta}(N)=N^{-\frac{\delta }{2}+\varepsilon} $ if $0<\delta \le 1$. 
 \end{thm}

\begin{rem}\label{rem;exp}
	We say $(X, \mu, h) $ is $k$-mixing with exponential rate if 	there exists $\sigma   >0$ such that the error term of (\ref{eq;letter}) is given by  $ O(e^{-\sigma \inf_{i\neq j}|n_i-n_j|})$. Clearly, exponential rate implies 
	$n^{-2}$ rate and hence Theorem \ref{thm;mixing} holds in this case. 
\end{rem}

\begin{rem}\label{rem;invertible}
	If we  assume  $h: X\to 
	X$ is invertible and (\ref{eq;letter}) holds for any $n_i \in \Z$ with $k=2\ell-1$, 
	then (\ref{eq;repeat}) holds almost surely with pairwise distinct nonzero   integers $m_1, \cdots , m_\ell$. 
	% and 
	%non-cluster sequence of integers $\{r_n \}$. 
\end{rem}

\begin{ex}\label{ex;nil}
	{\bf Nilmanifold automorphisms} Let $\Gamma$ be a lattice of a  simply connected nilpotent Lie group $G$. The homogeneous space $X=G/\Gamma$ is called a compact nilmanifold. An automorphism $h$ is a diffeomorphism of $X$ which lifts to an automorphism of $G$. 
	We assume that $h$ is ergodic with respect to the probability Haar measure $\mu$ on $X$. 
	
	We endow $G$ with a right invariant Riemannian metric and 
	consider $X$ as a metric space with the induced metric. 
	We fix $\theta>0$ and take 
	$\mathcal F=C^\theta(X)$ to  be the space of $\theta$-H\"older continuous  functions 
	on $X$. It follows from    \cite[Theorem 1.2]{gs14}  that $(X, \mu, h)$ is mixing with exponential rate. 
	In view of Remark \ref{rem;exp}, 
 Theorem \ref{thm;mixing} in the form of Remark \ref{rem;invertible} holds  for any $\ell \ge 1$. 
\end{ex}

\begin{ex}\label{ex;symbolic}
	{\bf Subshifts of finite type}
	Let $h$ be the left translation  on the space  $\Sigma=\{1, \ldots, m \}^\Z$ where $m\ge 2$.  We endow      $\Sigma $ with   a metric so that the distance of  two points $\bx=(x_i)$ and $\by =(y_i)$  are given by $\mathrm{d}(\bx, \by )=2^{-n}$, 
where $n$ is the largest nonnegative integer with $x_i=y_i$ for every  $|i|< n$. 
Let $A$ be an $m\times m$ aperiodic matrix with entries $0$ or $1$, i.e., there exists $n\ge 1$
such that all the entries of  $A^n$  are positive. Let 
\[
\Sigma_A=\{ \bx=(x_i)\in \Sigma: A_{x_i x_{i+1}}=1 \quad \forall i\in \Z   \}. 
\]
Then the dynamical system $(\Sigma_A , h)$ is topologically mixing, see \cite[Lemma 1.3]{bowen}. 
Let $\theta >0$ and $\mathcal F=C^\theta(X)$ be the space of $\theta$-H\"older continuous functions on $X$.  Let $\mu=\mu_f$ be the Gibbs measure given by some
  $f\in C^\theta(X) $, i.e., 
  $\mu$ is the unique $h$ invariant probability measure on $X$  maximizing  the pressure given by $f$, see \cite[Theorem 1.22]{bowen}. %We claim that the assumption of Theorem \ref{thm;mixing} is satisfied in this case for all $\ell \ge 1$. 
  The effective multiple mixing with exponential  rate was essentially proved by Kotani-Sunada \cite{ks} which uses a different definition of correlations. We will give a proof of effective multiple mixing   based on   \cite{ks}  in  \S \ref{sec;exp}. In view of Remark \ref{rem;exp}, the conclusion of Theorem  \ref{thm;mixing} in the form of Remark \ref{rem;invertible} holds for all $\ell \in \N $. 
  
  In particular, let $\nu$ be the probability measure on $\{1, \ldots, m \}$ given by a probability vector $(p_1, \ldots, p_m)$, i.e., $\nu(i)=p_i>0$. Then the measure $\mu=\nu^{\otimes \Z}$ is a Gibbs measure on $\Sigma$, see \cite[Theorem 9.16]{walters}. Therefore,   Theorem \ref{thm;mixing} implies a quantitative version of multiple ergodic theorem  for all the finite Bernoulli shifts. 
\end{ex}

It is easy to see that if the  quantitative multiple convergence holds for a measure preserving system 
$(X, \mu, h)$, then it holds for its factors. Here we say $(Y, \nu, g)$ is a factor of $(X, \mu, h)$, if there is a measurable map $\varphi: X\to Y$ such that $\varphi_*\mu=\nu$ and $\varphi h=g\varphi$. In particular, Theorem \ref{thm;mixing}  holds for
% one-sided  finite Bernoulli shifts and 
topologically mixing Anosov diffeomorphisms endowed  with Gibbs measures associated to smooth functions, where 
 we take $\mathcal F$ to be the space of smooth functions.

In certain non-invertible systems, effective  mixing decays polynomially with  coefficients given by the  $L^\infty$ norm of a test function. 
An important  family of dynamical systems having this property are   Young towers \cite{young}.
In this case  effective multiple mixing follows from effective mixing. As a corollary, one gets the quantitative multiple  pointwise  convergence. This is the content of the next theorem. 

\begin{thm}\label{thm;young}
 Let $(X, \mu , h)$ be a probability measure preserving system. Let $\mathcal L$ be a subspace of $L^2(X, \mu)$ containing constant functions 
 and let  $\|\cdot \|_{\mathcal L}$ be a norm on $\mathcal L$. Suppose there exist $\delta , C>0$ such that  for any $\varphi \in L^\infty(X, \mu)$, $\psi \in \mathcal L$ and $n\in \N$ we have
\begin{align}
\label{eq;young}
\left |\int_X\varphi(h^n x) \psi(x)  \dd\mu(x)-\int_X \varphi\dd \mu \int_X  \psi \dd\mu\right |\le C \|\varphi \|_{L^\infty }\|\psi \|_{\mathcal L}\, n^{-\delta}. 
\end{align}
Then for any $k\in \N$ the system  $(X, \mu, h)$ is $k$-mixing with rate $n^{-\delta}$ for the vector space $ \mathcal L\cap L^{\infty}(X, \mu)  $. 
Hence Theorem \ref{thm;mixing} holds. 
\end{thm}

%If $\delta \neq 1$, then we can replace $N^{\varepsilon}$ in (\ref{eq;better}) by   $\log ^{\frac{3}{2}+\varepsilon}N$.  %We will discuss some concrete systems in \S \ref{sec;polynomial}. 

Although the almost sure convergence theorems considered in the introduction are with respect to invariant measures, the general result Theorem \ref{thm;final} does not require any invariance of the measure, but only the  effective correlations. 
The   multiple pointwise convergence  considered in the appendix of 
\cite{shi} is of this type. 
Several results of effective multiple correlations are obtained by
Dolgopyat 
\cite[Theorem 2]{dolgopyat} for  partially hyperbolic systems.  The results of this paper 
(e.g.~Theorem \ref{thm;exponential})
can be applied to the settings there to get quantitative multiple pointwise convergence.

\textbf{Acknowledgements:}
We would like to thank Wen Huang,  Hanfeng Li and Weixiao Shen for drawing our attention to
\cite{dl}\cite{assani}, 
 \cite{walsh} and \cite{young} respectively. 
%Wen Huang for  drawing our attention to . 
%We would like to thank Hanfeng Li and Weixiao Shen  for their comments on the 
%the exposition.  

\section{General quantitative multiple convergence}

In this section we prove the general version of quantitative multiple convergence Theorem \ref{thm;final}. 
A simplified  version is given in Theorem \ref{thm;exponential}  under the assumption of effective  decay of correlations. The following theorem is the main tool to prove Theorem \ref{thm;final}.

\begin{thm}\label{thm;tool}
	Let $(Y, \nu)$ be a probability space and let $F_k: Y\to \R$ be a sequence of square integrable functions. Suppose there exist  $C,\sigma>0$ such that 
	for any nonnegative  integers $m< n$
	\begin{align}\label{eq;assumption}
	\int_{ Y }  \Big (\sum_{m< k\le  n} F_k(y)   \Big)^2\dd \nu(y)\le C( n^\sigma -m^\sigma).
	\end{align}
	Then given any $\varepsilon >0$, one has 
	\begin{align*}
	\lim_{ N\to \infty} \frac{\sum_{1\le k\le  N} F_k(y)}{ N^{\frac{\sigma}{2}}\log^{\frac{3}{2}+\varepsilon} N }=0
	\end{align*}
	for $\nu$-almost every $y\in Y$. 
\end{thm}
If $0<\sigma\le 1$ and $F_k(y)=f_k(y)-b_k$ where $f_k$ is nonnegative and 
$0\le b_k \le k^\sigma-(k-1)^\sigma$, then Theorem \ref{thm;tool} is a special case
of \cite[Chapter 1 Lemma 10]{sprindzhuk}. The proof of the theorem is essentially the same 
as that in \cite{sprindzhuk} based on dyadic decomposition, and the method is attributed to  Schmidt \cite{schmidt}. 
The extension  to the case where $1<\sigma<2$ allows us to prove quantitative ergodic theorems   
  in the case of polynomial decay of  correlations. On the other hand, the case where 
$\sigma\ge 2$ does not provide useful information in the application to the usual ergodic theorem, but it might be useful in other contexts.

\begin{proof}
	For nonnegative  integers $r\le s$, let $L_{s, r}$ be the  class of sets
	$\{m\in \N : i2^r<m \le (i+1) 2^r \}, (0\le i< 2^{s-r})$, and let $L_s=\bigcup_{0\le r< s} L_{s, r}$. For each positive integer $n< 2^s$, there exists a subset $L(n)$ of $L_s$
	with cardinality at most $s$ such that $\{1, 2, \ldots, n \}$ is a disjoint union
	of $I\in L(n)$, see \cite[Lemma 3.4]{ksw}. We take $s(n)$ to be the smallest integer $s$ such that $n<2^s$. 
	Then for any $\varepsilon, \sigma >0$ \begin{align}\label{eq;ks}
	\limsup_{n\to \infty} \frac{2^{s(n)\sigma/2}s(n)^{\frac{3}{2}+\varepsilon}}	{ n^{{\sigma}/{2}}\log^{\frac{3}{2}+\varepsilon} n }<\infty.
	\end{align}

For each positive integer $s$, in view of  (\ref{eq;assumption}) we have 
	\begin{align}\label{eq;total}
\int_Y	\sum_{I\in L_s}\big (\sum_{k\in I}  F_k(y)\big )^2\dd \nu(y)&= \sum _{0\le r< s}
\sum _{I\in L_{ s,r}}\int_Y\big (\sum_{k\in I}  F_k(y)\big )^2\dd \nu(y)
\le Cs 2^{\sigma s}. 
	\end{align}
		Let 
	\begin{align}\label{eq;ys}
	Y_s =\left\{y\in Y: \sum_{I\in L_s}\big (\sum _{k\in I } F_k(y) \big)^2 >s^{2+\varepsilon}2^{\sigma s}\right\}. 
	\end{align}
	By (\ref{eq;total}) and Chebyshev's inequality, 	we have  $\nu(Y_s)
	\le C s^{-1-\varepsilon}$.  
	Since $\sum_{s=1}^\infty\nu(Y_s)< \infty $, 
		the Borel-Cantelli lemma implies that there is a full measure set $Y_\varepsilon$  such that every $y\in Y_\varepsilon $ does not 
belong to  $ Y_s$ for $s$ sufficiently large. 

If $y\not \in Y_s$ and $n<2^s$, then by Cauchy-Schwarz inequality and (\ref{eq;ys}), we have
\begin{equation}
\label{eq;gang}
	\begin{aligned}
	\big(\sum_{k=1}^ n F_k(y)\big)^2 &=	\big(\sum_{I\in L( n)}\sum_{k\in I} F_k(y)\big)^2 
	\le s  \sum_{I\in L(n)}\big (\sum_{k\in I} F_k(y)\big)^2\\
	&\le s \sum_{I\in L_s} \big(\sum_{k\in I} F_k(y)\big)^2 
%	& \le s \sum_{I\in L_s } n_I^\sigma- m_I^\sigma \\
	 \le s^{3+\varepsilon} 2^{\sigma s}.
	\end{aligned}
\end{equation}	
For $y\in Y_\varepsilon$, by (\ref{eq;ks}) and  (\ref{eq;gang})
	\begin{align*}
\lim_{ n\to \infty} \frac{\left|\sum_{1\le k\le  n} F_k(y)\right|}{ n^{\frac{\sigma}{2}}\log^{\frac{3}{2}+\varepsilon} n }
= \lim_{ n\to \infty} \frac{\left| \sum_{1\le k\le  n} F_k(y)\right|}
{2^{s(n)\sigma/2}s(n)^{\frac{3}{2}+\frac{\varepsilon}{2}} } 
\frac{1}{s(n)^{\frac{\varepsilon}{2}}}
\frac{2^{s(n)\sigma/2}s(n)^{\frac{3}{2}+\varepsilon}}
{ n^{\frac{\sigma}{2}}\log^{\frac{3}{2}+\varepsilon} n } =0.
\end{align*}		
\end{proof}

Now we apply Theorem \ref{thm;tool} to the question of  multiple pointwise convergence. 
Let $X, Y$ be measurable spaces and  let 
  $\ell$ be a positive integer. For  $1\le i\le \ell$ and $n\in \N$,  let $\xi_i: Y\to X$, 
$h_{i, n}: X\to X$  and
$f_i: X\to \R $ be measurable maps. 
The $\ell$-multiple pointwise convergence considers
the average  of the map $F_n: Y\to \R$ defined by 
\[
F_n(y)=f_1 (h_{1, n}  \xi_1(y) ) f_2( h_{2,n}  \xi_2(y))\cdots f_\ell(h_{\ell,n} \xi_\ell(y))=\prod_{i=1}^\ell
f_i(h_{i, n} \xi_i(y)). 
\] 

Let $\nu$ be a probability measure on $Y$ such that all the $F_n\  (n\in \N)$
are square integrable. 
We assume that there exists $M, \delta >0$ and  a family of functions
 $b_{i, j}:\N\times \N  \to [1, \infty]$,
 $c_{i, j}:\N\to [1, \infty]$ where $i, j\in \{ 1,\ldots, \ell \}$  
such that 
\begin{equation}
\begin{aligned}
\label{eq;one}
\left |\int_Y F_m(y)F_n(y)\dd\nu(y)\right|&\le M\sum_{i,j=1}^\ell   \big(b _{i, j}(m, n) ^{-\delta }+c_{i,j}(m)^{-\delta}+c_{i,j}(n)^{-\delta}\big)
\end{aligned}
\end{equation} 
{for all }$m,n\in \N$. 
We make it convention that $\infty^{-\delta}=0$, and hence $b_{i, j}(m, n)=\infty$ means this term does not contribute to the upper bound estimate. We should interpret  (\ref{eq;one}) as
effective  $2\ell$-multiple correlations whose error term is given by each pair of functions or a single term.
The simplified setting in \S \ref{sec;exp} and (\ref{eq;kong}) may help the reader understand the  assumption 
(\ref{eq;one}). 
%(\ref{eq;kong}) may 
We  denote the cardinality of  a subset $S$ of $\N$ by 
 $|S|$.  

%where $  b _{i, j}(m, n)\ge 1$. We assume that there exists $M\ge 1$ such that 
\begin{thm}
	\label{thm;final}
	Let the notation and assumptions be as above. We assume that 
	for all $1\le i, j\le \ell$ and $ n\in \N $
	\begin{align}
	|\{k\in \N: c_{i, j}(k)\le n  \} |  \le Mn.
	\label{eq;three0} 
	\end{align}
For any  $1\le i, j\le \ell$,  we assume  either for all $m,n\in \N$
	\begin{align}
	\label{eq;two0}
		|\{k\in \N: b_{i, j}(m, k)\le n  \} |  \le Mn, 
	\end{align}
	or for all $m, n\in \N$
	\begin{align}
\label{eq;two1}
	|\{k\in \N: b_{i, j}(k, m)\le n  \} |  \le Mn.
\end{align}	
	Then given any $\varepsilon>0$,   for $\nu$-almost 
	every $y\in Y$ one has 
	\[
	\frac{1}{N}\sum_{k=1}^N F_k(y)=o_y(\rho_{\varepsilon,\delta}(N))
	\]
	where $\rho_{\varepsilon,\delta}(N)$ is  as in Theorem \ref{thm;mixing}.
	\begin{comment}
	\begin{enumerate}[label=(\roman*)]
		\item $\frac{1}{N}\sum_{k=1}^N F_k(y)=o_y(N^{-\frac{1}{2} }
		\log^{\frac{3}{2}+\varepsilon}N)$ when $\delta>1$;
		\item $\frac{1}{N}\sum_{k=1}^N F_k(y)=o_y(N^{-\frac{1}{2}+\varepsilon })$ when $\delta=1$;
		\item $\frac{1}{N}\sum_{k=1}^N F_k(y)=o_y(N^{-\frac{\delta}{2}}
		\log^{\frac{3}{2}+\varepsilon}N))$ when $0<\delta< 1$.
	\end{enumerate}
	\end{comment}
\end{thm}

We use the same constant $M$ in (\ref{eq;three0}), (\ref{eq;two0}) and (\ref{eq;two1}) for the sake of explicitness, and in applications it suffices to check that the left hand sides of them are   $\ll n$ with the implied constant independent of $m$ or $n$. 
Here and hereafter the notation $\varphi\ll \psi$ for two nonnegative functions means $\varphi\le C\psi$ for a constant $C>0$.
It is not hard to see that  (\ref{eq;three0}) holds if and only if  we can rearrange the sequence  $\{ c_{i, j}(n)  \}$ increasingly so that 
the integral part  of the $k$th term is larger than or equal to $\frac{k}{M}$.
Here the integral part of a real number  $s$ is the least integer bigger than or equal to $s$. We shall make use of this fact in the following way: given any subset $\{n_1,\dots,n_\ell\}\subset\N$, ordered in such a way that $c_{i,j}(n_k)$ is increasing for $\ell=1,\ldots ,k$, then $c_{i,j}(n_k)^{-1}\leq \frac{2M}{k}$.

We also note that there is the following sufficient condition: (\ref{eq;three0}) holds  if for any $s\in \N $
\begin{align}
\ |\{ k\ge 1:   c _{i, j}(k)\in [s, s+1]   \}|\le M. 
\label{eq;three}
\end{align}
%	where  the symbol $|\cdot|$  denotes the cardinality of a set. 
Similarly, 
 (\ref{eq;two0}) holds  if 
\begin{align}
|\{ k\ge 1:   b _{i, j}(m,k)\in [s, s+1]   \}|&\le M
\label{eq;two} 
%\mbox{for all } s\in \N, 0\le i, j\le \ell; 
\end{align}
for all $m\in \N$. A similar sufficient condition can be given for (\ref{eq;two1}).

\begin{proof}[Proof of Theorem \ref{thm;tool}]
%We show that  (\ref{eq;assumption}) holds for $\sigma=2-\delta$ and $C=$.
By (\ref{eq;one}), for any nonnegative integers $m<n$  we have
 \begin{equation}
 \label{eq;luobo}
\begin{aligned}
E(m, n)&:=
\int_Y\big (\sum_{m<k\le n} F_k(y)\big )^2\dd\nu(y)\\
&\le \sum_{m< s, t\le n} \left |\int_Y F_sF_t\dd\nu(y)\right|\\
&\le M\sum_{m< s, t\le n}\sum_{1\le i, j\le \ell }\big(  b _{i, j}(s, t) ^{-\delta }+c_{i,j}(s)^{-\delta}+c_{i,j}(t)^{-\delta}\big)\\
& \le2 M \sum_{1\le i, j\le \ell } \sum_{m< s\le n}\sum_{m<t\le n} \big(b _{i, j}(s, t) ^{-\delta }+c_{i,j}(t)^{-\delta}\big). 
\end{aligned}
 \end{equation}
By writing  $c_{i, j}(t)$ increasingly for $m<t\le n$ and using the equivalent description of (\ref{eq;three0}) before the proof, we have 
\begin{align}\label{eq;night}
\sum_{m<t\le n}c_{i, j}(t)^{-\delta} \le2^\delta \sum_{t=1}^{n-m} M^\delta t^{-\delta}. 
\end{align}
%For fixed $i, j$,  and $s$, if we order $b_{i, j}(s, t)$ increasingly for $m<t\le n$, then by (\ref{eq;two}), the $k$th term is bigger than or equal to   $k/M$. By (\ref{eq;three}), the same estimate
%holds for  
Suppose  (\ref{eq;two0}) holds for $b_{i,j}$, then similar estimate using (\ref{eq;two0})
gives 
\begin{align}
\label{eq;night0}
\sum_{m<t\le n} 
b _{i, j}(s, t) ^{-\delta }\le 2^\delta \sum_{t=1}^{n-m} M^\delta t^{-\delta}.
\end{align}

By (\ref{eq;night}) and 
(\ref{eq;night0}), we have 
\begin{equation}
	\label{eq;mihoutao}
	\begin{aligned}
\sum_{m<t\le n} \left(b _{i, j}(s, t) ^{-\delta }+c_{i, j}(t)^{-\delta}\right)&\le 2^{1+\delta} \sum_{t=1}^{n-m} M^\delta t^{-\delta}\le 2^{2+\delta }M^\delta \int_{1/2}^{n-m} t^{-\delta}\dd t\\
&\le
\left\{ 
\begin{aligned}
&2^{2+\delta}M^\delta (1-\delta)^{-1} (n-m)^{1-\delta} & \qquad\mbox{if }0<\delta <1\\
&2^{2+\delta}M \log2(n-m) & \qquad\mbox{if }\delta =1\\
&{2^{1+2\delta} M^\delta}{(\delta -1)^{-1}}& \qquad\mbox{if }\delta >1
\end{aligned}
 \right..
 \end{aligned}
\end{equation}
If (\ref{eq;two1}) holds for $b_{i, j}$, then one obtains similar estimate to (\ref{eq;mihoutao}).
 We assume without loss of generality that (\ref{eq;two0}) holds for all $b_{i, j}$. 

Using  (\ref{eq;mihoutao}) and (\ref{eq;luobo})
in different cases according to $\delta$, we will  have 
(\ref{eq;assumption}) for various $\sigma$. 
If $\delta>1$, then 
\[
E(m, n)\le 4^{1+\delta} M^{1+\delta}(\delta-1)^{-1}\ell ^2(n-m). 
\]
So (\ref{thm;tool})
holds for $\sigma=1$ and  the conclusion in this case follows 
from     Theorem \ref{thm;tool}. 

If $\delta=1$, then 
\begin{align}\label{eq;1}
E(m, n)\le 2^{3+\delta} M^{2}\ell ^2(n-m)
\log 2(n-m)\ll (n-m)^{1+{\varepsilon}},
\end{align}
where the implied constant depends on $M$, $\ell$ and $\vre$.  
We claim that for any nonnegative integers $m<n$, 
\begin{align}\label{eq;similar}
(n-m)^{1+{\vre}}\ll n^{1+{\vre}}-m^{1+{\vre}}, 
\end{align}
where the implied constant may depend on $\vre$ but not on $m$ or $n$. 
We prove the claim by considering two cases. 
If $m\le \frac{n}{2}$, then by the  mean value theorem, 
there exists $x\in [0, 1/2]$ such that 
\begin{align*}
(n-m)^{1+{\vre}}&=n^{1+{\vre}}(1-(m/n))^{1+{\vre}}=n^{1+{\vre}}(1-
(1+\vre)(1-x)^{\vre}m/n)\\
&\le n^{1+\vre}-(1+\vre)2^{-\vre}m n^\vre\le n^{1+{\vre}}-m^{1+{\vre}}. 
\end{align*}
If $n/2<m<n$, then 
\begin{align*}
(n-m)^{1+{\vre}}&=(n-m)(n-m)^{{\vre}}\le (n-m)(n/2)^\vre\le 2^{-\vre}(n^{1+\varepsilon}-m^{1+\varepsilon}).
\end{align*}
This completes the proof of the 
 claim. 
By (\ref{eq;1}) and (\ref{eq;similar}), we have (\ref{eq;assumption}) holds for $\sigma=1+{\vre}$ and  
 the conclusion in this case follows from  Theorem \ref{thm;tool}.

If $0<\delta<1$, then by (\ref{eq;luobo}), (\ref{eq;mihoutao})
and (\ref{eq;similar}), 
\begin{align}\label{eq;2}
E(m, n)\le 2^{3+\delta} M^{1+\delta}(1-\delta)^{-1}\ell ^2(n-m)^{2-\delta}
\ll n^{2-\delta}-m^{2-\delta}, 
\end{align}
where the implied constant is independent of $m$ or $n$. 
So (\ref{eq;assumption}) holds for $\sigma=2-\delta$ and 
the conclusion in this case follows from  Theorem \ref{thm;tool}. 

\end{proof}

\begin{comment}
In view of (\ref{eq;1}) and (\ref{eq;2}), we have for $0\le \delta\le 1$, \[
E(m, n)\ll (n-m)(n-m)^{\varepsilon/2}. 
\]

\begin{align}
& \le M^{1+\delta} \sum_{0\le i, j\le \ell } \sum_{m< s\le n}\sum_{1\le t\le n-m} t^{-\delta }\\
& \le  M^{1+\delta} \sum_{0\le i, j\le \ell } \sum_{m< s\le n}(n-m)^{1-\delta }\\
& \le M^{1+\delta }(\ell+1)^2(n-m)\log 2(n-m). 
\end{align}

\end{comment}

\section{Effective decay of correlations}\label{sec;exp}

The aim of this section is to formulate a simplified  version of Theorem \ref{thm;final} under the assumption of effective decay of correlations. 
This simplified  version can be applied directly  to prove  Theorem \ref{thm;mixing}. 
After that we prove Theorem \ref{thm;young}.
At the end of this section we  prove
Example \ref{ex;symbolic}.

To make our statement applicable to broad contexts and explain the assumption (\ref{eq;one}), 
we consider  a 
class of measurable transformations $\mathcal H$ on a probability space 
$(X, \mathcal X, \mu)$. We remark here that we don't assume elements of $\mathcal H$
preserve $\mu$. 

For $k\ge 1$, we use $\mathrm{dist}(t_0, t_1, \ldots, t_k)$  to denote the minimal distance between $k+1$ real numbers, i.e.,
\[
\mathrm{dist}(t_0, t_1, \ldots, t_k)=\min_{i\neq j}|t_i-t_j|. 
\]
%We remark here that we don't assume elements 
% of $\mathcal H$ commute with each other, nor they preserve the measure $\mu$, although in applications they do. 
We say the  action of $\mathcal H$ on
$
(X, \mu)$ has   $k$-multiple $n^{-\delta}$ decay of correlations
for  a vector space  of  measurable functions $\mathcal F_0$ if the following holds: there exists an injective map $t: \mathcal H\to \R\setminus \{ 0\}$
so that 
for any 
$h_1, \ldots, h_k\in \mathcal H$, 
%there exists  a family of nonzero real numbers $ t^{(i,j)}_i, t^{(i,j)}_j\ (1\le i,j\le k)$ 
% with  $t^{(i,j)}_i\neq - t^{(i,j)}_j $ such that for
any 
$f_1, \ldots, f_k\in \mathcal F_0$ and  
any $r_1, \ldots, r_k\in \N\cup \{0 \}$, one has 
\begin{align}\label{eq;exp}
\left| \int_X f_1 (h_1^{r_1} x)\cdots f_k(h_k^{r_k} x)\dd\mu(x)\right|\ll
\big(\mathrm{dist}(0,t(h_1)r_1, \ldots, t(h_k)r_k)+1\big )^{-\delta}
, 
\end{align}
where  the implied constant may depend on $f_1, \ldots, f_k, h_1, \ldots, h_k$ but  is independent of  $r_1,\ldots, r_k$.
 Although in this paper we only consider invariant measures, the formulation of effective  decay of correlations in (\ref{eq;exp}) is natural when $\mu$ is not invariant by the action, see \cite{shi}. 
Another aim of this general setting is to explain the assumption (\ref{eq;one}) in Theorem \ref{thm;final}.

\begin{thm}
	\label{thm;exponential}
	%	Let $(X,  \mu)$ be a probability space,  let $\mathcal H$ be a 
	% 	family  of measurable maps on $X$ and let  $\mathcal F$ be a class of commuting  measurable functions on $X$. 
	Suppose the action of $\mathcal H$ on
	$
	(X, \mu)$ has  $2\ell$-multiple  $n^{-\delta}$ ($\delta>0$) decay of  correlations for  
	a vector space of measurable  functions $\mathcal F_0$. Then for any $\varepsilon>0$, any pairwise  distinct maps $h_1, \ldots, h_\ell\in \mathcal H$, any $f_1, \ldots, f_\ell\in \mathcal F_0$ and any non-clustered sequence of positive integers $\{ r_n \}$, 
	one has $\mu$-almost surely
	\begin{align}\label{eq;good}
	\frac{1}{N} \sum_{n=1}^N f_1(h_1^{r_n} x)\cdots f_\ell(h_\ell^{r_n}x)= o_x(\rho_{ \varepsilon, \delta}( N)), 
	\end{align}
	where $\rho_{\varepsilon,\delta}(N)$ is  as in Theorem \ref{thm;mixing}. 
\end{thm}

%In this paper we only need to use Theorem \ref{thm;final}
%in the case where $Y=X$ and $\xi_i$ is the identity map. 
%The general setting will include  the result in \cite{shi}. 

\begin{proof}

 To  put
the  left-hand side of (\ref{eq;good}) into the framework of Theorem \ref{thm;final} we take 
    $Y=X$, $\xi_i=\mathrm{Id}, \nu=\mu$  and
$h_{i, n}=h_i^{r_n}$, where   $1\le i\le \ell$ and $ n\in \N$. 
	Let $t: \mathcal H \to \R\setminus \{0 \}$ be the injective map  given by the definition of $2\ell$-multiple 
	$n^{-\delta}$	decay of correlations, i.e., (\ref{eq;exp}) holds for $k=2\ell$. 
	Then      $t_i=t(h_i)\ (1\le i\le \ell)$ are  pairwise distinct and nonzero  real numbers.
	Recall that $\{r_n \}$ is a  non-clustered sequence of positive  integers. 
	
	For $1\le i ,j\le \ell$ and $m,n\in \N $, we take 
	\begin{equation}\label{eq;kong}
	\begin{aligned}
	b_{i,j}(m,n)&=	{|t_i r_m-t_j r_n|+1}, \\
	c_{i,j}(n)&=\left \{
	\begin{array}{ll}
|t_i r_n-t_jr_n|+1 & \mbox{if } i\neq j\\
 {|t_i r_n|+1 } & \mbox{if } i= j
	\end{array}
	\right.. 
	\end{aligned}
	\end{equation}
	It follows from (\ref{eq;exp}) with $k=2\ell$ that (\ref{eq;one}) holds
	with $\delta$ as in the statement and  some $M\ge 1$ independent of    $m$ or $n$. 
	We show that for any  $i,j\in \{1,2,\ldots, \ell \}$ one has  that (\ref{eq;three}) and  (\ref{eq;two})  hold, which   implies (\ref{eq;three0}) and (\ref{eq;two0}).
In view of Theorem \ref{thm;final}, this will complete the proof.

	 For (\ref{eq;two}),
		 it suffices to prove that the  cardinality of    $n\in \N $ satisfying 
	\begin{align}
	{t_j r_n-t_i r_m\in [s, s+1]} \Leftrightarrow 
	r_n\in t_j^{-1}[s+t_i r_m, s+1+t_i r_m]
	\end{align}
	is uniformly bounded for all $s\in \Z$ and $m\in \N$.
	Note that the interval $t_j^{-1}[s+t_i r_m, s+1+t_i r_m]$ has length $t_j^{-1}$ which is independent of $m$ or $s$. Since we assume the sequence $\{ r_n\}$ is non-clustered,  there exists $M\ge 1$ such that 
	\[
	|\{n\in \N: r_n \in I \}|\le 
	M
	\]
	for any closed  interval $I$ whose length is  no more than  $t_j^{-1}$. 
	Therefore (\ref{eq;two}) holds.	
For $i\neq j$, one can prove (\ref{eq;three}) similarly using
	 $t_i\neq t_j$.  The same argument using $t_i\neq 0$  proves (\ref{eq;three}) in the case where $i=j$. 
	
\end{proof}

\begin{proof}
	[Proof of Theorem \ref{thm;mixing}]
	Let 	$\mathcal F_0=\{f\in \mathcal F: \mu(f)=0 \}$, which is a subspace of $\mathcal F$
	since $\mathcal F$ contains constant functions. 
	We assume that $\mathcal F_0\neq \{0 \}$, otherwise the conclusion holds trivially. 	
	This assumption and (\ref{eq;letter}) imply $h^m\neq h^n$ for $m\neq n$.	
We take    
	$\mathcal H=\{ h^m: m\in \N \}$ and  $t(h^m)= m$. 
For any $k\in \N$, if   (\ref{eq;letter}) holds for $f_i\in \mathcal F$, then  (\ref{eq;exp}) holds for $f_i\in \mathcal F_0$. 
Since  we assume $\mathcal F$ contains constant functions, the $(2\ell-1)$-mixing  assumption implies   (\ref{eq;letter}) and hence (\ref{eq;exp})  for any $1\le k\le 2\ell-1$. 
  So
Theorem \ref{thm;exponential} can be applied in this setting with $\ell$ replaced by any element of  $\{1, \ldots, \ell \}$. 
A simple induction argument on $\ell$ shows that  equation 
(\ref{eq;repeat}) holds  for any $f_i \in \mathcal F$. 
The proof of the Remark \ref{rem;invertible} is the same with $\mathcal H=\{h^m: m\in \Z , m\neq 0 \}$.  

\end{proof}

%\begin{proof}
%	[Proof of Remark \ref{rem;invertible}]	If the assumption of Remark \ref{rem;invertible} holds for invertible $h$, then the proof is the same as that of Theorem \ref{thm;mixing}. The slight 	difference is that we take $\mathcal H=\{ h^m: m\in \Z, m\neq 0 \}$. \end{proof}

\begin{proof}
	[Proof of Theorem \ref{thm;young}]
	Since $\mathcal L\cap L^\infty(X, \mu)$ contains constant functions, by induction on $k$, it suffices to show that for any $k\in \N $, any  $f_0, \ldots, f_k\in \mathcal L\cap L^\infty(X, \mu)$ with $\mu(f_0)=0$ and any pairwise distinct natural numbers $n_1, \ldots, n_k$, 
we have 
	\begin{align}
	\label{eq;dajia}
	\left |\int_X f_0(x)f_1(h^{n_1}x)\cdots f_k (h^{n_k }x)\dd \mu(x) \right |\ll (\min \{n_1, \ldots, n_k \})^{-\delta}, 
	\end{align}
	where the implied constant depends on $f_0, \ldots, f_k$. 
To see this, we assume without loss of generality that $n_1<n_2<\cdots <n_k$ and 
 apply (\ref{eq;young}) for $\psi=f_0$,  $\varphi=f_1(x)f_2(h^{n_2-n_1}x)\cdots f_k(h^{n_k-n_1}x)$ and $n=n_1$. 
	Then we have the left hand side of (\ref{eq;dajia}) is 
	\[
	\le C \|f_1 \|_{L^\infty }\cdots \|f_k \|_{L^\infty }\|f_0\|_{\mathcal L}\, n_1^{-\delta}, 
	\]
	which implies (\ref{eq;dajia}). 
	
\end{proof}

\begin{proof}
	[Proof of Example  \ref{ex;symbolic}]
	  
	The effective multiple mixing  was proved in \cite{ks} with a different language.	
	Let  $k\in \N$ and 
	 $\mathbf f=(f_0, f_1, \ldots, f_k)$  where $f_i\in C^\theta(X)$.  
\cite{ks} considers 
	  \[
	\mathrm{Cov}(\mathbf f)
%§	=  \mathrm{Cov}(f_0, \ldots, f_k)
	:=\left. \frac{\partial^{k+1}}{\partial \theta_0\cdots\partial \theta_k} \log \left(\int_X 
	  \exp \big(\sum_{i=0}^k \theta_i f_i \big)
	  \dd\mu \right)\right|_{\theta_0=\cdots=\theta_k=0}.
	  \]
	  It follows from \cite[Proposition 3.1]{ks} that there exists $C, \sigma>0$ such that for any
	  integers $n_1,\ldots, n_k $ 
	  \begin{align}\label{eq;tired}
	 |	\mathrm{Cov}( f_0, f_1\circ h^{n_1}, \ldots,  f_k\circ h^{n_k})|\le C e^{-\sigma (\max_{i}|n_i|) }. 
	  \end{align}
	  Here $C$ may depend on $f_0, \ldots, f_k$, but we may choose one which works for a fixed finite family of functions  
	  and bounded $k$.   The independence of $\sigma $ from  $f_i$  is not explicitly stated in \cite[Proposition 3.1]{ks}, but it can be deduced from the proof.

	  For a nonempty subset   $I=\{i_0<\cdots <i_s \}$  of $\{0, \ldots, k \}$, we take  
	  $$\mathbf f_I=(f_{i_0}\circ h^{n_{i_0}}, \ldots , f_{i_s}\circ h^{n_{i_s}}).$$ By a partition $P=\{I_1, \ldots, I_s \}$ of $\{0, \ldots, k \}$, we mean a decomposition of $\{0, \ldots, k \}$ into a disjoint union of  elements of $P$. We use $\mathcal P(k)$ to denote the set of partitions of $\{0, \ldots, k \}$. By \cite[Lemma 2.1]{ks}, 
	  \begin{align}\label{eq;formular}
	  \int_X f_0\circ h^{n_0} f_1\circ h^{n_1}\cdots f_k\circ h^{n_k} \dd \mu=\sum_{P\in \mathcal P(k)}\prod_{I\in P}\mathrm {Cov}
	  (\mathbf f_I).
	  \end{align}
	  
\begin{comment}
	  Now we turn to the proof of (\ref{eq;exp}) with 
	     $k=2\ell$, $r_i\in \N$,   $\til f_i\in \mathcal F, h_i=h^{m_i}\ (1\le i\le k)$ for some $m_i\in \N$ and 
	      $t(h^n)=n\ (n\in \Z\setminus \{0 \})$.
	      We also take $r_0=m_0=0$ .
	    Using  (\ref{eq;formular}) with  $f_0=1$ and $f_i=\til f_i\circ h^{m_i r_i} \ (1\le i\le k)$, we have 
\[
\int_X \til f_1\cdots \til f_k\dd\mu=\sum_{P\in \mathcal P(k)}\prod_{I\in P}
\mathrm {Cov}
(\mathbf f_I).
\] 
Since the cardinality of $\mathcal P(k)$ depends on $k$, 
in view of (\ref{eq;exp}), it suffices  to prove that for any $P\in \mathcal P(k)$
\begin{align}\label{eq;noodle}
|\prod_{I\in P}\mathrm {Cov}
(\mathbf f_I)|\ll \sum_{0\le i<j\le k }
 e^{-\sigma|m_ir_i-m_jr_j|}
\end{align}
where we take $m_0=r_0=0$. 
Let $M=\sup \{ \til f_i(x): 1\le i\le \ell, x\in X \}$. The trivial upper bound of $|\mathrm{Cov}(\mathbf f_I)|$ is $M^{|I|}$. 

\end{comment}
If $I=\{i_0<i_1<\cdots< i_n \}$ where $n\ge 1$, then the $h$-invariance of  $\mu$ implies 
\[
\mathrm{Cov}(\mathbf f_I)=\mathrm{Cov}(  f_{i_0},
f_{i_1}\circ h^{n_{i_1} -n_{i_0}}, \cdots ,
 f_{i_s}\circ h^{n_{i_s}-n_{i_0}}).
\]
This together with (\ref{eq;tired}) implies 
\begin{align}\label{eq;hongkong3}
|\mathrm{Cov}(\mathbf f_I)|\le C\sum_{t=1}^s e^{-\sigma|n_{i_t} -n_{i_0}|}.
\end{align}
Note that the continuous  functions $f_i$ are bounded  on the compact space $X$. 
So if $P$ contains an $I$ with $|I|> 1$, in view of (\ref{eq;hongkong3}) 
	we have 
\begin{align}\label{eq;hongkong1}
\prod_{I\in P}|\mathrm{Cov}(\mathbf f_I)|\ll \sum_{t=1}^s e^{-\sigma|n_{i_t} -n_{i_0}|}.
\end{align}	
%Therefore (\ref{eq;noodle}) holds if $P$ contains some $I$ with $|I|> 1$.
 For the $P$ consists of only singletons, by the $h$-invariance of $\mu$,  we have 
\begin{align}\label{eq;hongkong2}
\prod_{I\in P}\mathrm {Cov}
(\mathbf f_I)=\prod_{i=0}^k \int_X f_i\dd\mu. 
\end{align}
Therefore the system $(X, \mu , h)$ is $k$-mixing with exponential rate for $C^\theta (X)$. 

\end{proof}

\section{subgroup actions on homogeneous spaces}
The aim of this section is to prove Theorem \ref{thm;unipotent}. 
Let the notation and assumptions   be as in Theorem \ref{thm;unipotent}. 
Recall that $H$ is a semi-simple Lie group  with finite center. 
We fix a maximal compact subgroup $K$ of $H$ and denote by $\Ad $ to be the adjoint 
representation. 
We endow the Lie algebra  $\mathfrak h$ of $H$  with an $\Ad(K)$ invariant inner product. This inner product induces 
a right invariant Riemannian metric $\rho$ on $H$
 and a Hermitian inner product on $\mathfrak h_\mathbb C=\mathfrak h\otimes \mathbb C$. 
For any $h\in H$, let  $\|\Ad (h)\|$ be the  operator norm on
 $\mathfrak h$ and $\mathfrak h_{\mathbb C}$, that is, 
\[
\|\Ad(h) \|=\max \{\|\Ad(h) v\|: v\in \mathfrak h, \|v \|=1  \}
=\max \{\|\Ad(h) v\|: v\in \mathfrak h_{\mathbb C}, \|v \|=1  \}.
\]
By \cite[Lemma 2.1(\rmnum{3})]{beg}, there exists $0<\kappa\le 1\le \tau$ such that for any $h\in H$
\begin{align}\label{eq;op}
\| \Ad (h) \| ^{\kappa}\ll \exp \rho(h,1_H) \ll \|\Ad(h) \|^{\tau}. 
\end{align}

\begin{proof}[Proof of Theorem \ref{thm;unipotent}]
It suffices to show that (\ref{eq;unipotent}) holds for $\mu$-almost every $x$ under the additional assumption that 
that $\mu(f_i)=0$ and  $f_i=\til f_i+c_i$ where  $\til f_i \in C_c^\infty(X) $ and $c_i \in \R$.  
We will apply Theorem \ref{thm;final}
to 
\[
F_n(x)= \prod _{i=1}^\ell f_i(h_i^{r_n} x),
\]
that is, in the case where $X=Y$, $\xi_i=\mathrm{Id}$ and $h_{i, n}=h_i^{r_n}$. 
We take \begin{equation}
\begin{aligned}
b_{i, j}(m,n)&=\|\Ad(h_i^m h_j^{-n})\|,\\
c_{i, j}(n)&= \left\{ 
\begin{array}{cc}
 \|\Ad(h_i^ n h_j^{-n})\|& \mbox{if }i\neq j  \\
\infty & \mbox{if }i=j
\end{array} \right..
\end{aligned}
\end{equation}
The existence of $M,\delta>0$ so that  (\ref{eq;one}) holds  follows from \cite[Theorem 1.1]{beg} and (\ref{eq;op}).

\begin{comment}

\begin{thm}
	\label{thm;beg}
	Let the notation and assumptions be as in Theorem \ref{thm;unipotent}. 
	Then there exists $\delta >0$ such that 
	\begin{align}\label{eq;beg}
	|F_n(x)F_m(x)|\le M \sum_{i, j} b_{i, j}(m,n)^{-\delta}. 
	\end{align}
\end{thm}
content...
\end{comment}

Now we show that (\ref{eq;three0}) holds for any $c_{i, j}$ with  $i\neq j$. 
The group generated by  $h=\Ad(h_ih^{-1}_j)$ is unbounded by the assumption. If $h$ is not quasi-unipotent, then there is an eigenvalue $r\in \mathbb C$ with $|r|>1$ for the action of $\Ad(h)$ on $\mathfrak h_{\mathbb C}$.  So
$
c_{i, j}(n)\ge |r|^n\gg n
$, which implies (\ref{eq;three0}). 
If $h$ is quasi-unipotent and generates an unbounded subgroup, then in view of the Jordan canonical form of $h$, there exists a two dimensional $h$-invariant subspace $V$ of $\mathfrak h_{\mathbb{C}}$ such that with respect to some basis of $V$
\begin{align}\label{eq;matrix}
h|_V=\left(
\begin{array}{cc}
s & 1 \\
0 & s
\end{array}
\right)\quad \mbox{ and }\quad h^n|_V=\left(
\begin{array}{cc}
s^n & n s^{n-1} \\
0 & s^n
\end{array}
\right)
\end{align}
where $|s|=1$. 
Therefore,
\[
\| h^n\|\ge \| h^n|_V\|\gg n,
\]
where the implied constant depends on the choice of the basis.
Therefore (\ref{eq;three0}) holds.

Next, we show that either (\ref{eq;two0}) or (\ref{eq;two1}) holds for each $b_{i,j}$. 
Let $h =\Ad (h_i)$ and $g=\Ad(h_j^{-1})$. The groups $\langle h\rangle$
and $\langle g\rangle $ are unbounded according to the assumption. 
We consider two cases. 

First, suppose both $h$ and $g$ are quasi-unipotent. 
We have the Jordan   
 decomposition $g=g_sg_u=g_ug_s$ where $g_s\in \GL(\mathfrak h)$ is semi-simple and $g_u\in \GL(\mathfrak h)$ is unipotent,   see \cite[\S I.4.4]{borel}. Since 
 $g$ is quasi-unipotent and $\langle g\rangle  $ is unbounded, we have 
 $\langle g_s\rangle $ is relatively compact and $g_u$ is non-trivial. Similarly, we have  the
 Jordan decomposition $h=h_sh_u$. Moreover, since $g$ and $h$ commute with each other, all the elements appearing in the  Jordan decomposition  commute with each other. Therefore, 
 there exists $C\ge 1$ such that  
 \[
 C^{-1}\| h_u^m g_u^n \|\le
  \|h^m g^n \| \le C \| h_u^m g_u^n \|. 
 \]
 So we can assume without loss of generality that $g$ and $h$ are unipotent. 
 By Engel's theorem, there is a basis of $\mathfrak h$ such that $g$ and $h$ are upper
 triangular unipotent matrices. There exists a non-diagonal entry    $s$ of $g$ such that 
 $s\neq 0$ and all the other  non-diagonal  entries of $g$ left to $s$ are zero, i.e.\
 one row of $g$ has the form 
 \[
 \left(
 \begin{array}{ccccccccc}	
 0&\cdots & 0 & 1 & 0 &   \cdots & 0& s & \cdots \\
 \end{array}
 \right).
 \]
 The same row of  $h^m g^n$ is
 \[
\left(
\begin{array}{ccccccccc}	
0&\cdots & 0 & 1 & 0 &   \cdots & 0& sn+t_m & \cdots \\
\end{array}
\right),
\]
where $t_m$ is the entry of $h^m$ at the same position of $s$. 
It follows that $\| h^m g^n\|\gg |sn+t_m|$, which implies (\ref{eq;two0}) with $M$ comparable to $s^{-1}$ but most importantly independent of  $m$.

In the second case we assume  at least one  of $g$ and $h$ is not quasi-unipotent. Suppose there exists
a simultaneous eigenvector
  $v\in \mathfrak h_{\mathbb C}$   of $g$ and $h$, say 
\[
gv =sv\quad \mbox{ and }\quad  hv =tv, 
\]
such that  $\max \{|s|, |t| \}>1 $ and $\min\{ |s| , |t| \}\ge 1 $.  Then either 
$
\| h^m g^n \|\ge (\max \{|s|, |t| \})^m 
$ or $
\| h^m g^n \|\ge (\max \{|s|, |t| \})^n  
$, which implies  that either (\ref{eq;two0}) or (\ref{eq;two1}) holds. Note that we can always find such an eigenvector $v$ if precisely one of $g$ and $h$ is quasi-unipotent: $v$ is taken to be a simultaneous  eigenvector associated to an eigenvalue of modulus strictly greater than one.

Suppose that none of the simultaneous eigenvectors have the property considered in the previous paragraph, then both of $g$ and $h$ are not quasi-unipotent. 
  Moreover,  every simultaneous eigenvector expanded  by $g$ is contracted by $h$ (and vice versa).
Let $S_+$ (resp.\ $S_-$) denote the set of  eigenvalues of $g$ with modulus $>1$  (resp.\ $<1$) appearing with multiplicity.   Similarly $T_\pm$ denote the corresponding eigenvalues  for $h$.  By assumption we have a bijection between $S_\pm$ and $T_\mp$, and we denote either map by $s\mapsto t_s$. 

Fix $m\in \N$. Let $k\in \N$ be the maximal positive   integer such that  $|s^nt_s^m|<1$ for all $n<k$ and all $s\in S_+$.
Let $s_+\in S_+$ for which $|s_+^kt_{s_+}^m|>1$. Since $g,h\in\SL(\mathfrak{h})$, there exists an eigenvalue $s_-\in S_-$ such that $|s_-^{(k-1)}t_{s_-}^m|>1$.
Therefore, 
\[
\|h^mg^n\|\ge \frac{1}{2}( |s_+^nt_{s_+}^m| + |s_-^nt_{s_-}^m|)\ge \frac{1}{2}\min\{|s^+ |, |t_{s^-}| \}^{|n-k|-1}. 
\]
Note that the lower bound above is an  exponential function whose base  is bigger than one and  independent of $m$, 
so (\ref{eq;two0}) holds.

If all the  $h_i$ and $h_i h_j^{-1} \ (i\neq j)$ are not quasi-unipotent, then $c_{i,j}(n)\ (i\neq j)$ and $b_{i,j}(m,n)$ are bounded from below by exponential functions whose base and coefficients are independent of  $m$. In this  case the bound  of (\ref{eq;one}) is the same as that in exponential decay
of correlations, so we have a better error term (\ref{eq;fuwocheng}). 
\begin{comment}
\[
\{n\in\N: \|h^mg^n\|\leq N\}\subset\{n\geq k:|s_+^nt_{s_+}^m|\leq N \}\cup \{n< k:|s_-^nt_{s_-}^m|\leq N \}.
\]
Considering the set for which $n\geq k$, we note that for such $n$ we have
$n\log{|s_+|}>m\log{|t_{s_+}^{-1}|}$, and the set therefore agrees with
\[
\left\{n\geq k: m\frac{\log{|t_{s_+}^{-1}}|}{\log{|s_+|}} < n \leq \frac{\log N}{\log{|s_+|}}+m\frac{\log{|t_{s_+}^{-1}}|}{\log{|s_+|}} \right\}.
\]

 %   such that $s_1^m t_1^n> 1 $ for $n> k$
%and $s_2^m t_2^n >1$ for $n<k$. 
\end{comment}

\end{proof}

\end{document}